\documentclass[12pt]{amsart}

\usepackage{amssymb,verbatim,ifthen}
\usepackage[all]{xy}
\def\split{true}

\ifthenelse{\equal{\split}{true}} {
\textwidth = 450pt \oddsidemargin = 2pt \evensidemargin = 2pt

} {
\usepackage{fullpage}

}
\newtheorem{theorem}{Theorem}

\newtheorem{lemma}{Lemma}

\newtheorem{proposition}{Proposition}

\newtheorem{defn}{Definition}
\newtheorem{example}{Example}

\newtheorem{remark}{Remark}

\def\1{{1\hskip-2.5pt{\rm l}}}
\def\eps{\varepsilon}

\def\PA{P_{\mathcal A}}
\def\VA{v_{\mathcal A}}
\def\VAn{v_{\mathcal A_n}}
\def\A{\mathcal A}
\def\F{\mathcal F}
\def\M{\mathcal M}
\def\eps{\varepsilon}

\begin{document}
\thanks{I am thankful to Yaron Azrieli, Ehud Lehrer and Chen Meiri for their helpful comments.
\\School of Mathematical Sciences, Tel Aviv University, Tel Aviv
69978, Israel. \\e-mail: ~teperroe@post.tau.ac.il}
\title{the induced capacity and choquet integral monotone convergence}
\author{Roee Teper}

\maketitle

\centerline{\today}
\bigskip
\bigskip\bigskip
\bigskip
\textbf{Abstract.} Given a probability measure over a state space, a
partial collection (sub-$\sigma$-algebra) of events whose
probabilities are known, induces a capacity over the collection of
all possible events. The \emph{induced capacity} of an event
$F$ is the probability of the maximal (with respect to inclusion)
event contained in $F$ whose probability is known. The
Choquet integral with respect to the induced capacity coincides
with the integral with respect to a \emph{probability specified on
a sub-algebra} (Lehrer \cite{Lehrer2}). We study Choquet
integral monotone convergence and apply the results to the
integral with respect to the induced capacity. The paper characterizes the properties of
sub-$\sigma$-algebras and of induced capacities which yield integral
monotone convergence.
\newpage

\section{Introduction}
In many economic situations individuals face uncertainties
regarding upcoming events. The probability of these events is
often unknown, and decision making is left to
subjective belief. The Ellsberg paradox \cite{Ell} demonstrates a
situation where (additive) expected utility theory (Savage \cite{Sav} and Anscombe and Aumann
\cite{Ans-Aum})  is violated due to partial information
obtained by the decision maker on the underlying probability.

Several proposed variations of the model relax the assumption of
additivity of the subjective probability (e.g.  Schmeidler
\cite{Sch}, Gilboa and Schmeidler \cite{Sch2}). Here we adopt a
recent model by Lehrer \cite{Lehrer2} which suggests a new
approach to decision making under uncertainty. The model describes
a decision maker who is partially informed about the underlying
probability. The information consists of the probability of some
(but maybe not all) events. According to Lehrer the decision maker
then assess her alternatives with only the information obtained
and  completely  ignores unavailable information.

What Lehrer actually suggests is a new integral for functions
which are not measurable  w.r.t. (with
respect to) the available information
given by a sub-$\sigma$-algebra. Given a probability measure space
$(X,\F, P)$ and a sub-$\sigma$-algebra $\A\subseteq \F$, the
integral of an $\F$-measurable function, say $f$, w.r.t.\  the probability $P$ restricted to $\A$,
 is the supremum
over all integrals of $\A$-measurable functions that are smaller
than or equal to $f$.

The integral induces a convex capacity over $\F$ as follows: the
capacity of an $\mathcal F$-measurable set $F$ is the integral of
the characteristic function of $F$, in other words, the
probability of the maximal (w.r.t.\ inclusion) $\A$-measurable set
contained in $F$. We call such a capacity the \emph{induced
capacity} by $\A$.

It turns out that the integral of an $\F$-measurable function
w.r.t.\ the restriction of $P$ to $\A$ coincides with the Choquet
integral (Choquet \cite{Choquet}) and the concave integral (Lehrer \cite{Lehrer1}) w.r.t.\ the induced
capacity.

The theory of the Lebesgue integral of sequences of functions is well
known. Choquet  and concave integral convergence theorems were also proved in
several versions (e.g. Li and Song \cite{Song}, Murofhushi and
Sugeno \cite {Suge2}, Lehrer and Teper \cite{Lehrer3}). In this analysis the precise definition of
what is `almost everywhere' w.r.t.\ a capacity is crucial. Several
definitions have been suggested (e.g. Klir and Wang \cite{Wang},
Lehrer and Teper \cite{Lehrer3}), studied and applied to integral
convergence theorems.

We focus on two definitions of `almost everywhere' convergence
w.r.t.\ a capacity which coincide with the usual definition in the
additive case. Utilizing these definitions of `almost everywhere' convergence
we prove new Choquet and concave integral monotone convergence theorems.

Applying integral monotone convergence, we study the convergence
of the integral w.r.t.\ the induced capacity. Sequences of
functions converging in different ways require different
properties of the induced capacity in order to obtain integral
convergence. Since the induced capacity is determined by a
sub-$\sigma$-algebra, different structures of the sub-$\sigma$-algebra
would yield convergence theorems for different types of convergent
sequences. We characterize these properties for several types of
convergence.
\\

The paper is organized as follows: Section $2$ presents the
notions of capacity, integration with respect to a capacity and `almost everywhere'
convergence. Then integral  monotone convergence is studied. In
Section $3$ we give the definition of an induced capacity by a sub
$\sigma$-algebra. The motivation behind the concept of the
induced capacity is brought in Section $4$. Sections $5$ and $6$
study the required properties of the sub-$\sigma$-algebra and of
the induced capacity which yield integral monotone convergence for
different types of converging sequences of functions. Finally,
discussion and comments appear in section $7$.

\section{Prelimanaries: Monotone Convergence of the Choquet
Integral}\label{sect_int_conv}

\subsection{The Choquet Intergral for Capacities}
Let $(X,\F)$ be a measurable space. A finite set function
$v:\mathcal F\rightarrow [0,\infty)$ is a \emph{capacity} if
$v(\emptyset)=0$ and  $v(F)\leq v(E)$ whenever $F\subseteq E$. A
capacity $v$ is convex (supermodular) if $v(F)+v(E)\leq v(F\cup
E)+v(F\cap E)$ for every $F,E$.

Denote by $\mathcal M$ the collection of all nonnegative
$\F$-measurable functions. The Choquet integral (see Choquet
\cite{Choquet}) of $f\in\mathcal M$ w.r.t.\ a capacity $v$ is
defined by
$$\int^{Cho}_Xfdv:=\int_0^{\infty}v(\{x :~ f(x)\geq t\})dt,$$ where
the latter integral is the extended Riemann integral. By the
definition of the Riemann integral,
$$\int^{Cho}_Xfdv=\sup\bigg\{\sum_{i=1}^N\lambda_i v(F_i):
 \sum_{i=1}^N\lambda_i \1_{F_i}\leq f\textrm{ and }\{F_i\}_i\in\mathcal F
 \textrm { is decreasing}
 ,\lambda_i\geq 0,N\in\mathbb N\bigg\},$$ where for every $F\in\F$, $\1_F$ is the
 characteristic function of $F$, and by decreasing we mean that $F_{i+1}\subseteq F_i$ for all
 $i<N$.

\subsection{Almost Everywhere Convergence} When discussing
sequences of functions, then ``almost everywhere'' convergence  arises
naturally. We study two different definitions for almost
everywhere convergence in the nonadditive case that coincide with
the usual definition in the additive case.

When a capacity is a measure, a sequence converges almost everywhere
if it converges over a set of full measure. Wang and Klir
\cite{Wang} proposed a  definition for almost everywhere convergence
when discussing a non-additive capacity $v$. According to their
 definition, a sequence  $\{f_n\}_n$ converges to $f$ $v$-a.e. iff
\\$v\left(\{x\in X : f_n(x)\nrightarrow f(x)\}\right)=0$. Since we
later define a stronger version of almost everywhere convergence,
whenever $v\left(\{x\in X : f_n(x)\nrightarrow f(x)\}\right)=0$ we
say that that $\{f_n\}_n$ converges to $f$ \emph{weakly $v$-a.e.\ }

Lehrer and Teper \cite{Lehrer3} introduced a stronger definition to
almost everywhere convergence.  We say that $\{f_n\}_n$ converges to
$f$ \emph{strongly $v$-a.e.\ }  iff $v\left(\{x\in F :
f_n(x)\rightarrow f(x)\}\right)=v(F)$ for all $F\in\mathcal F$. It
is also shown in \cite{Lehrer3} that if the capacity $v$ is convex,
then $v\left(\{x\in X : f_n(x)\rightarrow f(x)\}\right)=v(X)$
implies  that  $\{f_n\}_n$ converges to $f$ strongly $v$-a.e.

\begin{defn} A capacity $v$ is \emph{ null-additive} if $v(E\cup F)=v(F)$ for
every $E$ such that $v(E)=0$ and every $F$.
\end{defn}

\begin{lemma} \label{con_equi}Weak $v$-a.e.\ convergence coincides with strong
$v$-a.e.\ convergence iff $v$ is null-additive.
\end{lemma}
\begin{proof} Clearly, strong almost everywhere convergence implies
weak almost everywhere convergence.

Assume that $v$ is null-additive, that $\{f_n\}_n$ converges weakly
$v$-a.e.\ to $f$, and pick any $F\in\F$. Since $v$ is monotone,
$v(\{x\in F : f_n(x)\nrightarrow f(x)\})=0$.
\\Now, by null-additivity,  $$v(\{x\in F : f_n(x)\rightarrow
f(x)\})= v(\{x\in F : f_n(x)\rightarrow f(x)\}\cup \{x\in F :
f_n(x)\nrightarrow f(x)\})=v(F),$$ that is $\{f_n\}_n$ converges
strongly $v$-a.e.\ to $f$.

Conversely, assume that there exist $F,E\in\F$ such that $v(E)=0$
and $v(F\cup E)>v(F)$. Let $f=\1_{F\cup E}$ and $f_n=\1_F$ for all
$n$. Now, $v(\{x \in X : f_n(x)\neq f(x)\})=v(E)=0$ whereas
$v(\{x\in F\cup E : f_n (x)\rightarrow f(x)\})=v(F)<v(F\cup E).$
\end{proof}

\subsection{Monotone Convergence}
Li and Song \cite{Song} characterized capacities which satisfy
monotone Choquet integral convergence, when convergence of
sequences of functions is considered to be weak almost everywhere
convergence.

\begin{defn} A capacity $v$ is \emph{continuous from below} (resp. \emph{from above})
if $\lim_nv(F_n)=v(\bigcup_nF_n)$  (resp.
$\lim_nv(F_n)=v(\bigcap_nF_n)$) for every increasing (resp.
decreasing) sequence $F_1\subseteq F_2\subseteq \cdots $ (resp.
$F_1\supseteq F_2\supseteq \cdots $).
\end{defn}

\begin{theorem}[Li and Song]\label{choq_con} Let $v$ be a capacity. Then
$\lim_n\int^{Cho}_Xf_ndv=\int^{Cho}_Xfdv$ for any increasing
sequence $\{f_n\}_n$ converging weakly $v$-a.e.\ to $f$ iff $v$ is
null-additive and continuous from below.
\end{theorem}

The following theorem is a variant of the previous one, considering
strong almost everywhere convergence.

\begin{theorem}\label{choq_strong_con}Let $v$ be a capacity. Then
$\lim_n\int^{Cho}_Xf_ndv=\int^{Cho}_Xfdv$ for any increasing
sequence $\{f_n\}_n$ converging strongly $v$-a.e. to $f$ iff $v$ is
continuous from below.
\end{theorem}

The essence of the proof is in the next remark.

\begin{remark}\label{rem app}
Assume that $\{f_n\}_n$ is an increasing sequence  converging
strongly  $v$-a.e.\ to a function $f$. That is, $v(\{x\in F :
f_n(x)\rightarrow f(x)\})=v(F)$, for every $F\in\F$. If $v$ is
continuous from below, then for every $F\in\mathcal F$, $\eps'>0$ and
$\delta>0$, there is $N\in\mathbb N$ such that for every $n>N$,
$v(\{x\in F : f(x)-f_n(x)<\delta\})>v(F)-\eps'$.
\end{remark}

\begin{proof}[Proof of Theorem \ref{choq_strong_con}] If $v$ is not continuous from below then there
exist a sequence of increasing sets $\{F_n\}$ such that $\lim_nv(F_n)<v \big(\bigcup_nF_n \big)$ . Since
$\int^{Cho}_X\1_F=v(F)$ for every $F\in\F$, we have that $\lim_n\int^{Cho}_X\1_{F_n}<\int^{Cho}_X\1_F$.

For the converse implication, assume  first that  $\int^{Cho}_Xfdv<\infty$.
Since $f_n\leq f$,  $\lim_n\int^{Cho}_Xf_ndv\leq\int^{Cho}_Xfdv$. We
will show that for every $\eps>0$, there exist  $M$ such that for
every $n\geq M$, $\int^{Cho}_Xf_ndv>\int^{Cho}_Xfdv-\eps$.

Fix $\eps>0$. There exist  $\sum_{k=1}^N \lambda_k\1_{F_k}\leq f$
such that $\{F_k\}_k$ is decreasing and
$$\int^{Cho}_X fdv-\sum_{k=1}^K \lambda_kv({F_k})<\eps.$$

Applying  Remark \ref{rem app} to $F=F_k$,
$\eps'={{\eps}\over{K\lambda_k}}$ and
$\delta={{\eps}\over{v(X)K}}$ ($k=1,..., K$) one obtains an $N_k$
and a set $B_k=\{x\in F_k : f(x)-f_n(x)<{{\eps}\over{v(X)K}}\}$
that satisfy $v(B_k)>v(F_k)-{{\eps}\over{K\lambda_k}}$  for every
 $n\geq N_k$. Moreover, since $\{F_k\}_k$ is
decreasing and $\delta$ is fixed, then $\{B_k\}_k$ is decreasing
as well. Set $M:=\max\{N_1,...,N_K\}$. Now, for every $n\geq M$ we
get
$$\int^{Cho}_X
f_ndv>\sum_{k=1}^K\bigg(\lambda_k-{{\eps}\over{v(X)K}}\bigg)v({B_k})\ge
\sum_{k=1}^K \lambda_kv(B_k)-\eps>$$$$ \sum_{k=1}^K
\lambda_k\bigg(v({F_k})-{{\eps}\over{K\lambda_k}}\bigg)-\eps>\int^{Cho}_X
fdv -3\eps.$$ Since $\eps$ is arbitrarily small, the result
follows.

Now, if $f$ is not integrable, that is $\int^{Cho}_Xfdv=\infty$,
given a large $L$, there exist $\sum_{k=1}^K \lambda_k\1_{F_k}\leq
f$ such that
$$\sum_{k=1}^K \lambda_kv({F_k})>L.$$
The proof from this point is similar to the previous one.
\end{proof}

The following is a monotone convergence theorem for sequences of
functions that converge pointwise.

\begin{theorem}\label{cho_point_conv}Let $v$ be a capacity. Then
$\lim_n\int^{Cho}_Xf_ndv=\int^{Cho}_Xfdv$ for any increasing
sequence $\{f_n\}_n$ converging pointwise to $f$ iff $v$ is
continuous from below.
\end{theorem}

\begin{proof}
Continuity from below is necessary by Li and Song \cite{Song},
and sufficient by Murofushi and Sugeno \cite{Suge2}.
\end{proof}

\subsection{The Concave Integral and Monotone
Convergence}\label{cav_con}

The concave integral (see Lehrer \cite{Lehrer1}) of $f\in\mathcal
M$ w.r.t.\ a capacity $v$ is defined by
$$\int^{Cav}_Xfdv:=\sup\bigg\{\sum_{i=1}^N\lambda_i v(F_i):
 \sum_{i=1}^N\lambda_i \1_{F_i}\leq f\textrm{ and }F_i\in\mathcal F,\lambda_i\geq 0,N\in\mathbb N\bigg\}.$$
Clearly, $\int^{Cav}_Xfdv\geq \int^{Cho}_Xfdv$. It is shown in
Lehrer and Teper \cite {Lehrer3} that the concave integral coincides
 with the Choquet integral iff the capacity $v$ is convex.

The concave integral w.r.t.\ a capacity $v$ induced a
\emph{totally balanced cover} $\hat{v}$ over $\F$, which is a capacity itself.
The totaly balanced cover is defined by
$$\hat{v}(F):=\int^{Cav}_X\1_Fdv,    \textrm {  }\textrm {  }\textrm {  }\textrm {  for every }F\in\F.$$
The following lemma states  that in the view of the concave integral, all capacities are a totally balanced cover.
\begin{lemma}[Lehrer and Teper \cite{Lehrer3}]$\int^{Cav}_Xfdv=\int^{Cav}_Xfd\hat{v}$ for every $f\in\M$.
\end{lemma}

We now formulate monotone convergence theorems for the concave integral.
\begin{theorem}[Lehrer and Teper \cite{Lehrer3}]\label{conc_con}
$\lim_n\int^{Cav}_Xf_ndv=\int^{Cav}_Xfdv$ for any increasing
sequence $\{f_n\}_n \subset \mathcal M$ converging strongly
$v$-a.e. to $f$ iff $\hat{v}$ is continuous from below.
\end{theorem}

\begin{theorem}\label{con_weak_conv}
$\lim_n\int^{Cav}_Xf_ndv=\int^{Cav}_Xfdv$ for any increasing
sequence $\{f_n\}_n \subset \mathcal M$ converging weakly $v$-a.e.
to $f$ iff $\hat{v}$ is null-additive and continuous from below.
\end{theorem}

\begin{proof}The `only if' implication is clear. By Lemma \ref{con_equi} we have
that $\{f_n\}_n$ converges strongly $\hat{v}$-a.e. to $f$, therefore by
Theorem \ref{conc_con} we obtain the `if' implication.
\end{proof}

\begin{theorem}\label{con_weak_conv}
$\lim_n\int^{Cav}_Xf_ndv=\int^{Cav}_Xfdv$ for any increasing
sequence $\{f_n\}_n \subset \mathcal M$ converging pointwise to
$f$ iff $\hat{v}$ is and continuous from below.
\end{theorem}

\begin{remark}
The proof is similar to that of Theorem \ref{choq_strong_con}.
Instead of applying Remark \ref{rem app}, the reader should apply
the following argument: let $\{f_n\}_n$ be an increasing sequence
converging pointwise to a function $f$. If $\hat{v}$ is continuous
from below, then for every $F\in\mathcal F, \eps'>0$ and
$\delta>0$ there is $N\in\mathbb N$ such that for every $n>N$,
$\hat{v}(\{x\in F : f(x)-f_n(x)<\delta\})>\hat{v}(F)-\eps'$.
\end{remark}

\section{Sub-$\sigma$-algebra and the induced capacity} Let
$(X,\mathcal F,P)$ be a probability  space. A sub-$\sigma$-algebra
$\A\subseteq \F$ induces a convex capacity over $\F$  (see Lehrer
\cite{Lehrer2}) by
$$ v_{\mathcal A}(F)=\max\{P(A) : A\in\mathcal A,
A\subseteq F\},$$ for every $F\in\F$. $\A$ is a $\sigma$-algebra
therefore the maximum is attained and $\VA$ is well defined. We
denote by $A_F=\arg\max\{P(A) : A\in\mathcal A, A\subseteq F\}$
the $\mathcal A$-measurable set (modulo a set of probability $0$)
at which the maximum is attained. We say that $\VA$ is the
\emph{induced capacity} by $\A$.

\begin{remark}Since the induced capacity is convex, the Choquet and concave integral
w.r.t.\ it coincide. From this point forward, unless stated otherwise, the integral
w.r.t.\ the induced  capacity could be interpreted both as a Choquet and concave integral,
and the ``Cho'' and ``Cav'' notation is
therefore omitted.
\end{remark}

We now present the main interest of this paper.  The structure of a sub-
$\sigma$-algebra could be varied to induce capacities with
different properties. Now, assume that an
increasing sequence of measurable nonnegative functions $\{f_n\}_n
\subset \mathcal M$ converges in a certain way to a function $f$.
We would like to address the following questions:
\begin{itemize}\item  Does $\lim_n\int_Xf_nd\VA=\int_Xfd\VA$?
\item How to characterize the structure of a sub-$\sigma$-algebra
$\mathcal A$ which would yield such convergent sequence of
integrals? \item In what sense should sequences of functions
converge to obtain convergence of the integrals?
\end{itemize}

\begin{lemma}\label{null_2_cont} $\VA$  is continuous from
above.
\end{lemma}
\begin{proof}Assume that  $\{F_n\}_n$
is decreasing to $F$. Obviously, $A_F\subseteq A_{F_n}$ for every
$n$, therefore $A_F\subseteq \bigcap_n A_{F_n}$. Assume that
$\emptyset\neq A'= \bigcap_n A_{F_n}\setminus A_F\in\mathcal A$. In
particular, $A'\subseteq A_{F_n}$ for every $n$, therefore
$A'\subseteq A_F$, a contradiction.
\end{proof}

Continuity from below, which cannot be obtained for every induced
capacity,  plays a key property in integral convergence and will be
discussed in detail in Section \ref{sect_natom}.

\section{Motivation: Decision Making Under Uncertainty}
 Expected utility
is a customary theory  to analyze the behavior of a decision maker
(DM), where her preference order is described by the Lebesgue
integral. The Ellsberg paradox \cite{Ell} demonstrates a situation
where expected utility theory is violated due to partial
information that the DM obtains on the underlying probability. In
a recent paper,   Lehrer \cite{Lehrer2} suggests a new approach to
this issue. According to Lehrer, the preference order is given by
a new integral which utilizes only the information obtained by the
DM and ignores completely unavailable  information.

More formally, given a probability space $(X,\mathcal F,P)$ we
describe the information obtained by the DM by a
sub-$\sigma$-algebra $\A\subseteq \F$. The restriction of the
probability $P$ to $\A$, denoted by $\PA$, is called \emph{a
probability specified on a sub-algebra (PSA)}. The integral
w.r.t.\ a PSA $\PA$ of an $\mathcal F$-measurable nonnegative
function $f\in\mathcal M$  is defined by
$$\int_Xfd\PA=\sup\bigg\{\sum_{i=1}^N\lambda_i P(A_i):
 \sum_{i=1}^N\lambda_i \1_{A_i}\leq f\textrm{ and }A_i\in\mathcal A,\lambda_i\geq 0,N\in\mathbb N\bigg\}.$$

The next Lemma relates the integral w.r.t.\ a PSA to the induced
capacity by a sub-$\sigma$-algebra.

\begin{lemma} 1. $\VA(F)=\int_X\1_Fd\PA$ for all $F\in\F$.

2.\label{PSP_2_cap} $\int_XfdP_{\mathcal A}=\int_Xfd\VA$ for all
$f\in\mathcal M$.

\end{lemma}
\begin{proof} $1$ is straight forward. As for $2$, assume that $f$ is $\VA$-integrable and $\PA$-integrable.
Fix $\eps>0$. There exists $\sum_{n=1}^N\lambda_n\1_{F_n}\leq f$
such that $\int_Xfdv_{\mathcal A}\leq\sum_{n=1}^N\lambda_n
v_{\mathcal A}(F_n)+\eps$. Now,
$$\int_Xfdv_{\mathcal A}\leq \sum_{n=1}^N \lambda_nv_{\mathcal A}(F_n)+\eps=\sum _{n=1}^N \lambda_n P(A_{F_n})+\eps
 \leq\int_Xfd\PA+\eps.$$ In the same manner, there exist
$\sum_{n=1}^N\lambda_n\1_{A_n}\leq f$ such that
$\int_Xfd\PA\leq\sum_{n=1}^N\lambda_nP(A_n)+\eps$.
$$\int_Xfd\PA\leq
\sum_{n=1}^N\lambda_nP(A_n)+\eps=\sum_{n=1}^N\lambda_nv_{\mathcal
A}(A_n)+\eps\leq \int_Xfdv_{\mathcal A}+\eps.$$ Since $\eps$ is
arbitrarily small we obtain the expected result.

If, for example, $f$ is not $v_{\mathcal A}$-integrable, then for
every $L>0$ there exist $\sum_{n=1}^N\lambda_n\1_{F_n}\leq f$ such
that $\sum_{n=1}^N\lambda_n v_{\mathcal A}(F_n)>L$. The proof that
$f$ is not $P_\mathcal A$-integrable is similar to the one above.
\end{proof}

By Lemma \ref{PSP_2_cap} we can interpret the  integral w.r.t.\ a
PSA  as the Choquet integral w.r.t.\ to the induced capacity.

\section{Non-Atomic Probability Spaces}\label{sect_natom}

In this section we consider non-atomic probability
spaces.\footnote{$F\in\F$ is an atom if $P(F)>0$ and for every
$G\in\F$ such that $G\subseteq F$, $P(G)=P(F)$ or $P(G)=0$. A
probability space is non-atomic if there are no atoms.} Discrete
probability spaces will be discussed later in section
\ref{sect_atom}.

\subsection{Weak Almost Everywhere Convergence} When considering
weak almost everywhere convergence, then Theorem \ref{choq_con} (and
Theorem \ref{con_weak_conv}) states that integral (Choquet and
concave) monotone convergence is equivalent to null-additivity and
continuity from below.
\begin{lemma}\label{null_2_cont}If $\VA$ is null-additive then it is continuous from
below.
\end{lemma}
\begin{proof}Indeed, assume that $\VA$ is null-additive and let $\{F_n\}_n$ be a
sequence of measurable sets increasing to $F$ such that
$\lim_n\VA(F_n)<\VA(F)$. Set $C_n=\left(A_F\setminus \bigcup_n
A_{F_n}\right)\cap F_n$. $\{C_n\}_n$ is increasing to
$\left(A_F\setminus \bigcup_n A_{F_n}\right)$ and $\VA(C_n)=0$ for
every $n$. Now, set $D_n= \left(A_F\setminus \bigcup_n
A_{F_n}\right)\setminus C_n$. $\{D_n\}_n$ is decreasing to the
emptyset, and by continuity from above $\VA(\bigcap_n D_n)=0$.
However, by null-additivity $\VA(D_n)=\VA \left(A_F\setminus
\bigcup_n A_{F_n}\right)=P\left(A_F\setminus \bigcup_n
A_{F_n}\right)$ which is positive.
\end{proof}

\begin{defn}  We say that a collection $\mathcal C\subseteq \mathcal
F$ is \emph{dense} in $\mathcal F$ iff for every $\eps>0$ and $F\in
\mathcal F$ there exist  $C\in\mathcal C$ such that $C\subseteq F$
and $P(F\backslash C)<\eps$.
\end{defn}

Since  $\mathcal A$ is a $\sigma$-algebra, then being dense in
$\mathcal F$ is equivalent to that, for every $F\in\mathcal F$, there
exist $A\in\mathcal A$ included in $F$ such that $P(F\backslash
A)=0$.

\begin{proposition}\label{weak_equi}The following are equivalent:
\\1. $\mathcal A$ is dense in $\mathcal F$;
\\2. $\int_Xfd\VA=\int_XfdP$ for every function
$f\in \mathcal M$;
\\3. $\lim_n\int_Xf_nd\VA= \int_Xfd\VA$
for every increasing sequence of functions $\{f_n\}_n\subset
\mathcal M$ converging weakly $\VA$-a.e. to a function $f$; and
\\4. $\VA$ is null-additive.
\end{proposition}

\begin{proof}
$(1)\Leftrightarrow (2)$. Pick $F\in\mathcal F$. If $\mathcal A$ is
dense in $\mathcal F$ then $\VA(F)=\max\{P(A) : A\in\mathcal A,
A\subseteq F\}=P(F)$, therefore the integrals coincide.

Now assume that $\int_Xfd\VA=\int_XfdP$ for every
 function $f\in\mathcal M$. In particular, for all $F\in\mathcal F$,
$\int_X\1_Fd\VA=\int_X\1_FdP=P(F)$. But $\int_X\1_Fd\VA$ is equal to
the probability $P(A)$ of some $A\in\mathcal A$ contained in $F$.
Thus $\mathcal A$ is dense in $\mathcal F$.

$(1)\Rightarrow(3)$ is simply  Levi's monotone convergence theorem.

$(3)\Leftrightarrow(4)$  by  Theorem \ref{choq_con} (or Theorem
\ref{con_weak_conv}) and Lemma \ref{null_2_cont}.

$(4)\Rightarrow(1)$. Assume that $\mathcal A$ is not dense in
$\mathcal F$. Then there exist $F\in\mathcal F$ of positive
probability such that $A_F=\arg\max\{P(A) : A\in \mathcal A,
A\subset F\}$ satisfies $P(A_F)<P(F)$. Denote  $B=A_F\cup F^c$.
 $P(B^c)=P(F\setminus A)>0$ therefore $P(B)<1$, thus $\VA(B)<1$.
 However, $\VA(B^c)=\VA(F\setminus A)=0$, and by null additivity
 $\VA(B)=\VA(B\cup B^c)=\VA(X)=1$. A contradiction.

\end{proof}

\subsection{$P$ - Almost Everywhere Convergence}
\begin{defn} A sub-$\sigma$-algebra $\A$ satisfies property $(A1)$  if for every $F\in\mathcal F$ with
$P(A_F)>0$ there exist $\delta>0$ such that for every $G\in\mathcal
F$ contained in $F$ which $P(F\setminus G)<\delta$ satisfies
$P(A_G)>0$.
\end{defn}

\begin{example}\label{cont_not_*} Consider the unit interval $[0,1]$, $\mathcal B^1$ the Borel
$\sigma$-algebra over $[0,1]$ and $m^1$ the Lebesgue probability
measure. Let $C\subset [0,1]$ be a set of probability $0$ and of
continuum cardinality. There exist a one-to-one correspondence
$f:C\rightarrow [0,1]\setminus C$. Now, let $\A=\{A \subseteq [0,1]
: x\in A \textrm{ iff } f(x)\in A\}$. The induced capacity is
continuous from below, however (A2) is not satisfied. Indeed,
$m^1_{\A}([0,1])=1$ and $m^1_{\A}([0,1]\setminus C)=0$ where
$m^1([0,1]\setminus ([0,1]\setminus C))=0$. That is $\A$ does not
satisfy property (A1).
\end{example}

\begin{example}\label{exmp_not_dense} Let $\mathcal A$ be the sub-$\sigma$-algebra which
contains all sets $A\in\mathcal B^1$ which satisfy
$A=A+{{1}\over{2}}$ (mod $1$). Note that $\mathcal A$ is not dense
in $\mathcal B^1$ since, for example, ${m^1}_{\mathcal
A}([0,1/4])=0$. For $F\in\mathcal B^1$, $A_F=F\cap (F+
{{1}\over{2}}$ (mod $1$)$)$. Now, assume that $\{F_n\}_n$ is
increasing to $F$. We will show that $\{A_{F_n}\}_n$ is increasing
to $A_F$ which will prove that $\VA$ is continuous from below.
Indeed, if $x\in A_F$ then $x\in F\cap (F+ {{1}\over{2}}$ (mod
$1$)$)$. That is, there exist $n\in\mathbb N$ such that $x\in
F_n\cap (F_n+ {{1}\over{2}}$ (mod $1$)$)$, meaning that $x\in
A_{F_n}$, as desired. It follows by Proposition \ref{weak_equi} that
${m^1}_{\mathcal A}$ is not null-additive . For example,
${m^1}_{\mathcal A}([0,1/2])=0$ and ${m^1}_{\mathcal A}([1/2,1])=0$,
whereas ${m^1}_{\mathcal A}([0,1])=1$. Furthermore, $\A$ satisfies
 (A2). For $F\in\mathcal F$ with $P(A_F)>0$, set
$\delta={{P(A_F)}\over{4}}$.
\end{example}

\begin{defn}An induced capacity $\VA$ is \emph{$P$-null-additive}
if $\VA(G)=\VA(F)$ for every $G,F$ such that $G\subseteq F$ and
$P(F\setminus G)=0$.
\end{defn}

\begin{proposition}\label{a.e_con} The following are equivalent:
\\1. $\mathcal A$ satisfies property (A1) ;
\\2. $\lim_n\int_Xf_nd\VA= \int_Xfd\VA$
for every increasing sequence of functions $\{f_n\}_n\subset
\mathcal M$ converging $P$-a.e. to a function $f$; and
\\3. $\VA$ is continuous from below and $P$-null-additive.
\end{proposition}

\begin{proof} $(1) \Leftrightarrow (3)$.  Assume that $\VA$ is continuous from below and $P$-null-additive. Assume that there exist $F\in\mathcal F$ with
$P(A_F)>0$ such that for every $\delta>0$ there exist
$G\in\mathcal F$ contained in $F$ such that $P(F\setminus
G)<\delta$ and $P(A_G)=0$. Pick a sequence $\{\delta_n\}_n$ such
that $\delta_n\rightarrow 0$, then there is a sequence
$\{G_n\}_n\subset\mathcal F$ such that $G_n\subseteq F$,
$P(F\setminus G_n)<\delta_n$ and $P(A_{G_n})=0$ for all $n$.
$\1_{G_n}$ converges to $\1_F$ in probability $P$, therefore there
exist a subsequence $\1_{G_{n_m}}$ that converges to $\1_F$
$P$-almost everywhere.  The sequence $\{H_m\}_m$ where
$H_m=\bigcap_{k\geq m}G_{n_k}$ is increasing and $P(A_{H_m})=0$
for every $n$. Set $\widetilde{F}=\bigcup_mH_m$. Showing that
$P(A_{\widetilde{F}})>0$ will establish that $\VA$ is not
continuous from below. Since $\VA$ is $P$-null-additive and
$P(F\setminus \widetilde{F})=0$, we obtain that
$\VA(\widetilde{F})=\VA(F)>0$, as desired.

Conversely, assume that there exist $G\subseteq F$ such that
$P(F\setminus G)=0$ and $\VA(F)>\VA(G)$. Then $A_F=A_G\uplus E
\uplus H$ where $E\subseteq F\setminus G, H\subseteq G$. Now,
$\VA(E\uplus H)=P(H)>0$ where $\VA(H)=0$ and $P(E)=0$, therefore
(A1) does not hold. Furthermore, if there is a sequence $\{F_n\}_n$
increasing to $F$ such that $\lim_n \VA(F_n)<\VA(F)$. We obtain that
$\lim_n P(A_{F_n})<P(A_F)$, in particular, $\bigcup_n
A_{F_n}\subsetneq A_F$. Note that  $A'=A_F \setminus \bigcup_n
A_{F_n}\in\mathcal A$, $\VA(F_n\cap A')=0$  for all $n$ and
$\VA(F\cap A')>0$. $P(A_{F'})>0$ where $F'=F\cap A'$. The sequence
$\{F'_n\}_n$ where $F'_n=F_n\cap A'$ increases to $F'$, thus since
$\mathcal A$ satisfies (A1) there is $n$ large enough so that
$P(A_{F'_n})>0$, a contradiction.

$(2) \Leftrightarrow (3)$. Assume first that $\lim_n\int_Xf_nd\VA=
\int_Xfd\VA$ for every increasing sequence of functions
$\{f_n\}_n\subset \mathcal M$ converging $P$-a.e.\ to a function $f$.
The continuity from below of $\VA$ is obvious. If $G\subseteq F$
such that $P(F\setminus G)=0$, then $P(\1_G=\1_F)=1$ therefore
$\VA(G)=\VA(F)$, and we obtain weak null-additivity as well.

Conversely, let $\{f_n\}_n$ be an increasing sequence converging
$P$-a.e.\ to a function $f$. That is, $P(\{x\in F :
f_n(x)\rightarrow f(x)\})=P(F)$, for every $F\in\F$. If $\VA$
assume that $\VA$ is continuous from below and $P$-null-additive,
then for every $F\in\mathcal F, \eps'>0$ and $\delta>0$ there is
$N\in\mathbb N$ such that for every $n>N$, $v(\{x\in F :
f(x)-f_n(x)<\delta\})>v(F)-\eps'$. From this point the proof is
similar to that of Theorem \ref{choq_strong_con} (note that since
the integral w.r.t.\ an induced capacity is the concave integral,
using a collection of decreasing sets is not necessary).

\end{proof}

\subsection{Strong Almost Everywhere Convergence}
Since an induced capacity $\VA$ is convex, then
$\int_X\1_Fd\VA=\VA(F)$ for all $F\in\mathcal F$. Thus, Theorem
\ref{choq_strong_con} (and  Theorem \ref{conc_con}) can be applied
to an induced capacity whenever strong almost everywhere convergence
is in hand. The theorem states that monotone convergence holds for
the integral w.r.t.\ an induced capacity iff it is continuous from
below.

\begin{example}\label{exmp_not_cont}Consider $X=[0,1]^2$, $\mathcal B^2$ the Borel
$\sigma$-algebra over $[0,1]^2$ and $m^2$ the Lebesgue probability
measure. Let $\mathcal A=\{[0,1]\times B : B\in \mathcal B^1\}$.
Consider any sequence of the form $\{B_n\times [0,1]\}_n$ where
$\{B_n\}_n$ is increasing to $[0,1]$. ${m^2}_{\mathcal A}(B_n\times
[0,1])=0$ for all $n$ whereas ${m^2}_{\mathcal A}([0,1]^2)=1$. That
is $m^2_{\A}$ is not continuous from below.
\end{example}

\begin{defn} A sub-$\sigma$-algebra $\A$ satisfies property $(A2)$ if for every $F\in\F$
such  that $P(A_F)>0$ and every $\{F_n\}_n$ increasing to $F$ there
is $n$ such that $P(A_{F_n})>0$.
\end{defn}

\begin{proposition}The following are equivalent:
\\1. $\mathcal A$ satisfies property (A2) ;
\\2. $\lim_n\int_Xf_nd\VA= \int_Xfd\VA$
for every increasing sequence of functions $\{f_n\}_n\subset
\mathcal M$ converging strongly $\VA$-a.e. to a function $f$;
\\3. $\lim_n\int_Xf_nd\VA= \int_Xfd\VA$
for every increasing sequence of functions $\{f_n\}_n\subset
\mathcal M$ converging
 pointwise to a function $f$; and
\\4. $\VA$ is continuous from below.
\end{proposition}

\begin{proof}
$(1)\Leftrightarrow (4)$. Clearly continuity from below implies
(A2). As for the other implication, assume that $\{F_n\}_n$ is
increasing to $F$ and that $\lim_n\VA(F_n)<\VA(F)$. Setting
$A'=A_F\setminus \bigcup_nA_{F_n}$, then $\VA(A')=P(A')>0$. Denote
$F'_n=F_n\cap A'$ for all $n$, then $\{F'_n\}_n$ is increasing to
$A'$ and $\VA(F'_n)=0$ for all $n$, that is $\VA$ is not continuous
from below.

$(2)\Leftrightarrow (4)$ by Theorem  \ref{choq_strong_con} (and Theorem
\ref{conc_con}).

$(3)\Leftrightarrow(4)$ by Theorem \ref{cho_point_conv}.
\end{proof}

To conclude this section, which  discusses monotone convergence of
the integral w.r.t.\ the induced capacity in non-atomic probability
spaces, we present the following diagram which summarizes the
properties presented.
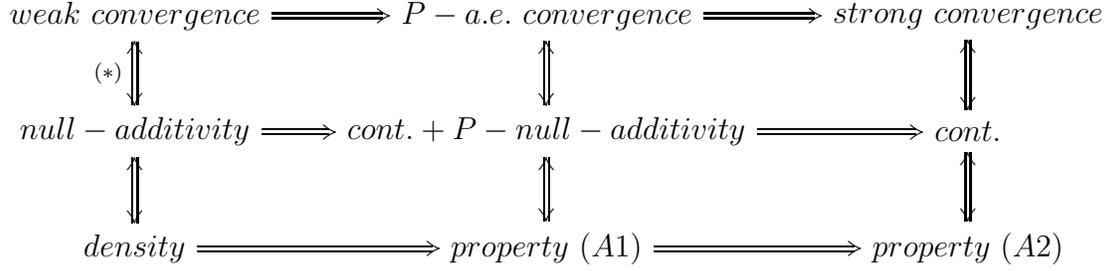
\begin{displaymath}
\xymatrix{ weak\textrm { } convergence  \ar@2[r] \ar@2[d]_{(*)}  &  P-a.e.\textrm{ }convergence  \ar@2[d]       \ar@2[r]                   &  strong\textrm{ } convergence  \ar@2[d] \\
           null-additivity  \ar@2[d] \ar@2[r]  \ar@2[u]    &  cont.+ P-null-additivity \ar@2[d] \ar@2[u] \ar@2[r]        &  cont. \ar@2[u]  \ar@2[d] \\
           density   \ar@2[u]      \ar@2[r]                &  property\textrm { }(A1) \ar@2[u]               \ar@2[r]                &  property\textrm { }(A2)  \ar@2[u]
}
\end{displaymath}

The top row shows to which type of converging sequences of
functions monotone convergence holds. The middle row indicates the
appropriate induced capacity which will suffice for the relevant
type of convergence. The bottom row states the property of the
sub-$\sigma$-algebra which would yield a corresponding  property of the
induced capacity.  For instance, the arrow marked with $(*)$ is
simply the consequence of Proposition \ref{weak_equi} that monotone
convergence of the integral w.r.t.\ an induced capacity holds for
every weak almost everywhere convergent sequence if and only if the
induced capacity is null-additive.

\section{Discrete Probability Spaces}\label{sect_atom}

Consider the case where  $X$ is a countable (possibly finite) space
endowed with some probability $P$. For the sake of convenience, if
$|X|=n$ then we will assume that $X=\{1,...,n\}$, otherwise
$X=\mathbb N$. Here $\mathcal A$ is some $\sigma$-algebra generated
by a partition of $X$, $\{A_i\}_{i\in\mathbb N}$. Namely,
$\{A_i\}_{i\in\mathbb N}$ are the atoms of $\mathcal A$.

Note that in this case \begin{equation}\label
{count_int}\int_Xfd\VA=\sum_i\bigg(\inf_{x\in A_i} f(x)\bigg)
P(A_i)\end{equation} for all $f\in\mathcal M$.

\begin{example} \label{exmp_conv} Let $P(k)\thickapprox {{1}\over{k^2}}$ for every $k\in\mathbb N$, and  $\mathcal
A$ be the $\sigma$-algebra generated by the partition $\{\{2k,2k-1\}
: k\in\mathbb N\}$.

Let $f=1$
 and for every $n$
\begin{displaymath}f_n(k)=\left\{\begin{array}{ll}
1, & \textrm{$k\leq n$,}\\%
0, &   \textrm{$k> n$.}
\end{array}\right.
\end{displaymath}

By (\ref{count_int}) $\int_{\mathbb N}fd\VA=1$ and $\int_{\mathbb
N}f_nd\VA=\sum_{k\leq n}P(k)$ where the later converges to $1$.

Denote by   $\mathcal T=(\emptyset,\mathbb N)$ the trivial field.
$\int_{\mathbb N}fdP_{\mathcal T}=1$ but since $\min f_n=0$ for all
$n$, $\int_{\mathbb N}f_ndP_{\mathcal T}=0$ for all $n$.
\end{example}

Example \ref{exmp_conv} might only reflect two particular structures
of $\mathcal A$. In the first example all atoms of $\mathcal A$ are
finite and we obtain a sequence of integrals which converge to the
integral of the limit function. In the second example there is an
infinite atom and we are unable to obtain integral convergence. The
following proposition shows that in fact there are only two cases.

\begin{proposition} The following are equivalent:
\\1. $\mathcal A$ is generated by atoms
consisting of finitely many elements of $X$;
\\2. $\lim_n\int_Xf_nd\VA= \int_Xfd\VA$ for every increasing sequence of functions $\{f_n\}_n$
converging pointwise to a function $f$;
\\3. $\lim_n\int_Xf_nd\VA= \int_Xfd\VA$ for every increasing sequence of functions $\{f_n\}_n$
converging strongly $v$-a.e. to a function $f$; and
\\4. $\VA$ is continuous from below.
\end{proposition}

\begin{proof} $(1) \Rightarrow(2)$. Assume at first that $f$ is
$\VA$-integrable. Let $\{f_n\}_n$ be any increasing sequence
converging to $f$ and fix $\eps>0$. Let $n_*\in\mathbb N$ be big
enough so that $\sum_{i>n_*}\big(\inf_{x\in A_i}
f(x)\big)P(A_i)<\eps$. Define $N_1:=min\{n\geq 1
:f(m)-f_n(m)<{{\eps}\over{n_*P(A_1)}}, m\in A_1\}$. By induction,
define $N_{i}:=min\{n\geq N_{i-1}
:f(m)-f_n(m)<{{\eps}\over{n_*\inf_{A_i} f P(A_i)}}, m\in A_i\}$ for
all $2\leq i\leq n_*$. Since every $A_i$ is finite, it is guaranteed
that $N_i$ are finite for every $i\leq n_*$. Now,
$$\int_X f_nd\VA\geq \sum _{i\leq n_*}\left(\inf_{x\in A_i} f(x)\right) P(A_i)-\eps\geq
\int_X fd\VA-2\eps.$$ Since $\eps$ is arbitrarily small, we have
that $\int f_nd\VA\geq  \int fd\VA$. The inverse inequality is
obvious.

If $f$ is not $\VA$-integrable, that is $\int_Xfd\VA=\infty$, given
a large $L$, there exist $n_L$ $\sum_{i\leq n_L}\big(\inf_{x\in A_i}
f(x)\big)P(A_i)>L$. The proof from this point is similar to the one
above.

$(2)\Rightarrow (4)$ by the definition of continuity from below.

$(4) \Rightarrow (1)$. Assume that there exist an atom
$A=\{k_1,k_2.,,,\}$ with infinite  number of elements of $X$. Let
$A_n=\{k_1,\dots,k_n\}$ for every $n$.  $\VA(A)=P(A)$ whereas
$\VA(A_n)=0$ for all $n$.

$(3)\Leftrightarrow (4)$ by Theorem \ref{conc_con}.

\end{proof}

\begin{remark}By Theorem \ref{choq_con}, integral convergence
where functions converge weakly almost everywhere is obtained iff
the induced capacity is null-additive.  To obtain null-additivity of
the induced capacity it is easy to see that $\A$ must be generated
by the singletons of $X$, that is $\A=2^X$.
\end{remark}

\section{Discussion and Final Comments}

\subsection{Generalizing Induced Capacities}
Lehrer \cite{Lehrer2} considers a second model of decision making
with \emph{partially-specified probabilities (PSP)}. The PSP model
illustrates the case where a DM obtains the information of the
integrals of a sub collection $\mathcal G \subseteq \mathcal M$ of
nonnegative measurable functions. The DM makers then approximates
the integral of $f\in\mathcal M$ by the supremum over all positive
combinations of integrals of functions in $\mathcal G$ that are
smaller than or equal to $f$. Formally, the integral of  $f$ w.r.t.\
$\mathcal G$ is defined by
$$\int_XfdP_{\mathcal G}=\sup\bigg\{\sum_{i=1}^N\lambda_i \int_Xg_idP :
 \sum_{i=1}^N\lambda_i g_i\leq f\textrm{ and }g_i\in\mathcal G,\lambda_i\geq 0,N\in\mathbb N\bigg\}.$$
The PSP model is indeed a generalization of the PSA model, since
one could consider $\mathcal G$ to be the collection of
characteristic functions of some sub-$\sigma$-algebra $\A$.

In the case of PSP, the analogous definition for the induced
capacity is
$$v_{\mathcal G}(F):=\int_X\1_FdP_{\mathcal G}$$ for every $F\in\F$.
Properties of the induced capacity in the PSA model do not  hold
for the induced capacity in the PSP model. For example
$v_{\mathcal G}$ is no longer convex, thus the Choquet and concave
w.r.t.\ it do not coincide. It follows that  Lemma \ref{PSP_2_cap}
is no longer true. Integration type needs to be specified.

This discussion raises several questions: For which collections of
functions would the induced capacity would be convex? For which
collections of functions could Lemma \ref{PSP_2_cap} be formulated
for the Choquet and concave integral? It would also be interesting
to study the properties of such collections that yield integral
convergence theorems.

\subsection{Increasing Information}\label{increase_subsect}
Assume that, at each period of time, a DM obtains more information
regarding the underlying probability. That is $\{\A_n\}_n$ is an
increasing sequence of sub-$\sigma$-algebras. We consider the case
where the union of $\{\A_n\}_n$ generates a dense sub
$\sigma$-algebra of $\F$. For all $n$ denote by $v_n=\VAn $ the
induced capacity by $\A_n$. We would like to know whether $\lim
_n\int_Xfdv_n=\int_XfdP$. $\{v_n\}_n$ is an increasing sequence of
capacities. We say that that it increases continuously to $P$, if
$\lim_nv_n(F)=P(F)$ for all $F\in\F$.

\begin{lemma}$\lim _n\int_Xfdv_n=\int_XfdP$ for every $f\in \mathcal M$ iff $\{v_n\}_n$ increases continuously to $P$.
\end{lemma}
\begin{proof}
The `only if' direction is obvious. As for the `if' direction,
assume that $\{v_n\}_n$ increases continuously to $P$ and fix
$\eps>0$. $\int_XfdP\leq \sum_{i=1}^N \lambda_iP(F_i) +\eps$ for
some  $\sum_{i=1}^N \lambda_i\1_{F_i}\leq f$. For all $i\leq N$
there exist $N_i$ such that for every $n\geq N_i$ $v_n(F_i)\geq
P(F_i)-{{\eps}\over {N\lambda_i}}$. Denoting  $M=\max_i N_i$, we
obtain that
$$\int_XfdP\leq \sum_{i=1}^N \lambda_iP(F_i) +\eps\leq \sum_{i=1}^N \lambda_iv_n(F_i) +2\eps\leq \int_Xfdv_n+2\eps$$
for all $n\geq M$. Since $\eps$ is arbitrarily small we have that
$\int_XfdP\leq \lim_n\int_Xfdv_n$. The converse inequality is
trivial.
\end{proof}
For the next $2$ examples consider the Borel $\sigma$-algebra over
$[0,1]$ endowed with the Lebesgue measure.
\begin{example} \label{inc_not_cont}   For every $n$, $\A_n$ is the algebra generated by the diadic partition of length $2^{-n}$.
The union of $\A_n$ generates the Borel $\sigma$-field. Now, let
$F$ be the set of all irrationals in $[0,1]$. $P(F)=1$ where
$v_n(F)=0$ for every $n$.
\end{example}

\begin{example} Let $\A_n$ be the $\sigma$-algebra generated by all Borel measurable sets contained in $[0,a_n)$  and
the set $[a_n,1]$, where $\{a_n\}_n$ is increasing  to $1$. It is
clear that $\{v_n\}_n$ increases continuously to $P$.

\end{example}
\begin{remark}\label{general_cont} $\{v_n\}_n$ increases continuously to $P$ iff the union of $\{\A_n\}_n$ is  dense in $\F$.
\end{remark}

In light of Lemma \ref{PSP_2_cap}, a DM preference order obtaining
partial information, as abundant as it might be, could be
completely different from that of a DM obtaining the complete
information. Consider Example \ref{inc_not_cont}. A fully informed
DM would prefer $\1_F$ to $\1_{F^c}$, whereas a DM who is informed
of $\A_n$ have no preference.

\subsection{Families of Functionals} We have seen in Section
\ref{sect_int_conv} that for both
Choquet and concave integrals, in order to obtain integral
monotone convergence, the capacity needs to satisfy the same
properties  (considering of course a specific type of a converging
sequence of functions).

Which properties of the Choquet and concave integral are essential
for obtaining monotone convergence for the exact same capacities? In
other words, assume that one defines a new functional w.r.t.\ $v$
over $\mathcal M$. Denote it by $I$. What characterizes $I$ so that
monotone convergence would occur for the exact same properties of
the capacities that yield monotone convergence for the Choquet and
concave integrals?


\begin{thebibliography}{99}

\bibitem{Ans-Aum} Anscombe, F. J. and Aumann, R. J. (1963) A definition of subjective
probability, \textit{Annals of Mathematical Statistics}, 34,
199-295.

\bibitem{Sch2} Gilboa, I. and Schmeidler, D. (1989) Maxmin expected utility non-unique prior, \textit{Journal of Mathematical Economics}, 18, 141-153.

\bibitem{Choquet} Choquet, G. (1955) Theory of
capacities, \textit{Ann. Inst. Fourier}, 5, 131-295.

\bibitem{Ell} Ellsberg, D. (1961) Risk, Ambiguity, and the Savage axioms, \textit{ Quarterly Journal of Economics}, 75, 643-669.


\bibitem{Wang} Klir, G. J. and Wang, Z.  (1992) \emph{ Fuzzy
measure theory,} {Plenum, New York}.

\bibitem{Lehrer1} Lehrer, E. (2005) A new integral for
capacities, \emph{Economic Theory}, to appear.

\bibitem{Lehrer2} Lehrer, E. (2006)  Partially-specified probabilities: decisions and games, mimeo.

\bibitem{Lehrer3} Lehrer, E. and Teper, R. (2007)  The concave integral over large spaces, mimeo.

\bibitem{Sav} Savage, L. J.  (1954) \emph{ The foundations of tatistics,} {Wiley, New York}.

\bibitem{Sch} Schmeidler, D. (1989) Subjective probabilies without additivity, \textit{Econometrica}, 57, 571-587.

\bibitem{Suge2} Murofushi, T. and Sugeno, M.   (1991) A theory of fuzzy measures: representations, the Choquet integral and
null-sets, \textit{J. Math. Anal. App.}, 159, 532-549.

\bibitem{Song} Li, J. and Song, J.   (2005) Lebesgue theorems in non-additive
measure theory, \textit{Fuzzy sets and systems}, 149, 543-548.

Japan.
\end{thebibliography}
\end{document}